\documentclass[letterpaper,11 pt]{article}
\usepackage{graphicx}
\usepackage{amsfonts,amsmath,fullpage,bbm}
\usepackage{amssymb,amsthm,multirow,verbatim}
\usepackage{acronym,wrapfig,plain,mathrsfs,enumerate,relsize,color}
\newtheorem{algorithm}{Algorithm}
\usepackage{algorithm}
\newtheorem{theorem}{Theorem}

\newtheorem{lemma}{Lemma}

\newtheorem{assumption}{Assumption}
\newtheorem{definition}{Definition}
\usepackage{subfig}
\usepackage{algorithmic}
\usepackage{pifont}
\usepackage{footmisc}

\newcommand{\aj}[1]{{\color{black}#1}}

\providecommand{\keywords}[1]{\textbf{\textit{keywords:}} #1}
\allowdisplaybreaks
\begin{document}
\allowdisplaybreaks
\title{Primal-Dual Incremental Gradient Method for Nonsmooth and Convex Optimization Problems
}


\author{Afrooz Jalilzadeh\footnote{ Department of Systems and Industrial Engineering, The University of Arizona, 1127 E James Rogers Way, Tucson, AZ. Email: afrooz@arizona.edu}}




\date{}

\maketitle
\begin{abstract}
In this paper, we consider a nonsmooth convex  finite-sum problem with a conic constraint. To overcome the challenge of projecting onto the constraint set and computing the full (sub)gradient, we introduce a primal-dual incremental gradient scheme where only a component function and two constraints are used to update each primal-dual sub-iteration in a cyclic order. 
We demonstrate an asymptotic sublinear rate of convergence in terms of suboptimality and infeasibility which is an improvement over the state-of-the-art incremental gradient schemes in this setting.  
Numerical results suggest that the proposed scheme compares well with competitive methods.

\end{abstract}
\keywords{Incremental Gradient; Primal-Dual Method; Convex Optimization.}
\section{Introduction}
\label{intro}
Convex constrained optimization has a broad range of applications in many areas, such as machine learning (ML). As data gets more complex and the application of ML algorithms becomes more diversified, the goal of recent ML research is to improve the efficiency and scalability of algorithms. In this paper, we consider the following nonsmooth and convex constrained  problem,
\begin{align}\label{p1}
&\min_{x\in X} f(x)\quad \mbox{s.t.} \quad Ax-b\in -\mathcal K, 
\end{align} 
where $f(x)=\sum_{i=1}^mf_i(x)$, $A=\begin{bmatrix}A_1^T&\hdots&A_m^T\end{bmatrix}^T$, $b=\begin{bmatrix}b_1^T&\hdots&b_m^T\end{bmatrix}^T$, $\mathcal K=\Pi_{i=1}^m \mathcal K_i$.  For each $i\in\{1,\hdots,m\}$, the function $f_i : \mathbb R^n \to \mathbb R$ is convex (possibly nonsmooth), $A_i\in \mathbb R^{d_i\times n}$, $b_i\in \mathbb R^{d_i}$ and $\mathcal K_i\subseteq \mathbb R^{d_i}$ is a closed convex cone, and $X\subset \mathbb R^n$ is a compact and convex set. \aj{We assume that the projection onto $\mathcal K_i$ can be computed efficiently while the projection onto the preimage set $\{x\mid A_ix-b_i\in -\mathcal K_i\}$ is assumed to be impractical for any $i\in\{1,\hdots,m\}$.  }

Letting $d\triangleq \sum_{i=1}^m d_i$, we introduce a dual multiplier $y=[y_i]_{i=1}^m\in\mathbb R^d$ for the constraint in  \eqref{p1} and $y^*$ denotes a dual optimal solution. Suppose a constant $B>0$ exists such that $\|y^*\|\leq B$. Such a bound $B$ can be computed efficiently if a slater point of \eqref{p1} is available, see Lemma \ref{slater}. Let $Y=\Pi_{i=1}^m Y_i$, $Y_i=\{y_i\in \mathbb R^{d_i}\mid\sqrt{m} \|y_i\|\leq B+1\}$ which implies that $\|y\|\leq B+1$, for all $y\in Y$, then problem \eqref{p1} can be equivalently written as the following saddle point (SP) problem:
\begin{align}\label{sp}
&\min_{x\in  X}\max_{y=[y_i]_{i=1}^m\in Y\cap\mathcal K^*} \phi(x,y)\triangleq\sum_{i=1}^m f_i(x)+y_i^T(A_ix-b_i),
\end{align} 
where $\mathcal K^*$ denotes the dual cone of $\mathcal K$, i.e., $\mathcal K^*\triangleq \{u\in \mathbb R^d: \langle u,v\rangle\geq 0,\ \forall v\in \mathcal K \}$.  

\paragraph{Motivation.} Problem \eqref{p1} has a broad range of applications in ML, wireless sensor networks, signal processing, etc. Next, we illustrate examples written in the form of problem \eqref{p1} in which projecting on the constraint is challenging. \vspace{-4mm}
\paragraph{Example 1.} (Basis Pursuit Denoising (BPD) problem) Let $x^*$ be a solution of a linear system of equations $Ax=b$, where $A$ and $b$ represent a transformation matrix and the observation vector, respectively. This problem arises in signal processing, image compression and compressed sensing \cite{chen1994basis} to recover a sparse solution $x$ given $A$ and $b$. In particular, one needs to solve 
$\min_x\{\|x\|_1 \mid  A_ix=b_i, ~ \forall i\in\{1,\hdots,m\}\}$. In real-world applications, the observations $b$ might be noisy \cite{donoho2006compressed}. Therefore, the problem can be formulated as follows: 
$$\min \ \|x\|_1, \  \mbox{s.t.} \ \|A_ix-b_i\|\leq \delta/\sqrt{m}, \ \forall i\in\{1,\hdots,m\}.$$
BDP problem is a special case of \eqref{p1} and the constraint  can be written as $(b_i-A_ix,-\delta/\sqrt{m})\in -\mathcal K_i$ where $\mathcal K_i=\{(y,t)\in \mathbb R^{d_i}\times \mathbb R\mid \|y\|\leq t\}$ is a second-order cone. Projection onto the second-order cone can be computed as:
\begin{align*}
\mathbf \Pi_{\mathcal K}(y,t)=\begin{cases} (y,t)& \mbox{if}\  \|y\|\leq t;\\ (0,0) & \mbox{if} \ \|y\|\leq -t; \\ \frac{\|y\|+t}{2}\left(\tfrac{y}{ \|y\|},1\right)& \mbox{otherwise},\end{cases}
\end{align*}
where $\mathbf \Pi_\mathcal K(y)$ denotes the projection of $y$ onto $\mathcal K$ \cite{bauschke1996projection}; however, projection onto the preimage set $\{x\mid A_ix-b_i\in -\mathcal K_i\}$ can be impractical.\vspace{-4mm}
\paragraph{Example 2.} (Constrained Lasso problem)  Let $y\in \mathbb R^{s}$, $B\in \mathbb R^{s\times n}$, and $x\in \mathbb R^n$ denote the response vector, the design matrix of predictors, and the vector of unknown regression coefficients. Then the general Lasso problem can be written as follows \cite{gaines2018algorithms}:
\begin{align*}
\min_x \tfrac{1}{ 2}\|y-Bx\|^2+\lambda_1 \|x\|_1+\lambda_2\sum_{j=2}^n |x_j-x_{j-1}|, \ \mbox{s.t.} \ Ax=b, \ \mbox{and} \ Cx\leq d,
\end{align*}
where $\lambda_1,\lambda_2\geq 0$  are the tuning parameters. 
The constrained Lasso problem above is a special case of problem \eqref{p1} by defining $f_i(x)=\tfrac{1}{ 2m}\|B_ix-y_i\|^2+\tfrac{\lambda_1}{ m}\|x\|_1+\tfrac{\lambda_2}{ m}\sum_{j=2}^n |x_j-x_{j-1}|$, where $B=[B_1^T,\hdots, B_m^T]^T$ and $y=[y_1^T,\hdots,y_m^T]^T$. The constraint can be written as $\begin{bmatrix} A_i\\ C_i\end{bmatrix} x-\begin{bmatrix} b_i\\ d_i\end{bmatrix}\in -\mathcal K,$ where $A_i\in \mathbb R^{p_i\times n}$, $C_i\in \mathbb R^{q_i\times n}$, $p_i+q_i=d$, $\mathcal K= \{\mathbf 0_{p_i}\}\times \mathbb R^{q_i}_+$ and $\mathbf \{\mathbf 0_{p_i}\}\in \mathbb R^{p_i}$. \vspace{-4mm}

\paragraph{Related work.} One of the main approaches to solve problem \eqref{p1}, when the projection is cheap, is using the \aj{Projected Incremental Gradient (PIG) scheme \cite{nedic2001incremental} where  the (sub)gradient of the function is approximated in a deterministic manner and cyclic order. Let $C\triangleq \{x\in X\mid Ax-b\in - \mathcal K\}$ denote the constraint set in \eqref{p1}, then each iteration of PIG has the following main steps:\\
{\bf for} $i=1,\hdots, m$\\ \vspace{-3mm}

1. Set $x_{k,1}=x_k$ and pick stepsize $\gamma_k$;\\\vspace{-3mm}

2. $x_{k,i+1}=\mathbf \Pi_C (x_{k,i}-\gamma_kg_{k,i});$\\\vspace{-3mm}

3. Set $x_{k+1}=x_{k,m+1}$,\\\vspace{-3mm}

\noindent {\bf end}}

\aj{where $g_{k,i}\in \partial f_{i}(x_{k,i})$, and $\partial f_i(x)$ denotes subdifferential of function $f_i(x)$, for all $i\in\{1,\hdots,m\}$}. When the problem is nonsmooth and convex, the convergence rate of $\mathcal O(1/\sqrt k)$ has been shown for PIG. The accelerated variant of Incremental Gradient (IG) scheme is studied in \cite{blatt2007convergent,defazio2014saga,gurbuzbalaban2017convergence,le2013stochastic}. These methods require storing a variable of size $\mathcal O(mn)$ at each iteration, hence, are impractical for large-scale problems and/or when the projection is hard to compute. One avenue to handle the constraints is by leveraging iterative regularization schemes \cite{amini2019iterative,yousefian2017smoothing}. Recently in \cite{kaushik2020projection}, authors introduced averaged iteratively regularized IG method that does not involve any hard-to-project computation to solve and require storing a variable of size $\mathcal O(n)$. However, their  suboptimality and infeasibility rates are $\mathcal O(1/k^{0.5-b})$ and $\mathcal O(1/k^b)$, respectively, for some $b\in(0,0.5)$. In contrast to the existing methods, in this paper, we address the challenge of projection by introducing a primal-dual scheme requiring memory of $\mathcal O(n+d/m)$. Moreover, our new primal-dual IG scheme improves the rate results to $\mathcal O(1/\sqrt k)$ in terms of suboptimality and infeasibility. 

Convex constrained optimization problems can be viewed as a special case of saddle point problems using Lagrangian duality. Different primal-dual methods have been introduced to solve such problems. Consider a saddle point problem of the form $\min_{x\in X}\max_{y\in Y} f(x)+\phi(x,y)-g_i(y_i)$, where $\phi(x,y)=\sum_{i=1}^m \langle A_ix-b_i,y_i \rangle$. 
When the objective function is strongly-convex strongly-concave and smooth, a linear convergence rate has been shown in \cite{yu2015doubly,zhang2017stochastic,zhu2015adaptive} using stochastic methods by randomly selecting the dual and/or primal coordinates. Assuming a merely convex-concave setting, the convergence rate of $\mathcal O(1/k)$ has been shown in \cite{chambolle2018stochastic,xu2017first}. Moreover, 
Xu \cite{xu2020primal} considered problem \eqref{p1} with nonlinear constraint $h_i\leq 0$ where $h_i$ is convex, and bounded function and $\partial h_i$ is bounded. They proposed a stochastic augmented Lagrangian scheme with convergence rate of $\mathcal O(1/\sqrt k)$. In this paper, we aim to recover the rate of $\mathcal O(1/\sqrt k)$ by approximating the subgradient in a deterministic manner and considering weaker assumptions. Finally, in our recent work \cite{jalilzadeh2019doubly}, we considered $\min_x \max_y \sum_{i=1}^m f_i(x_i)+ \sum_{j=1}^p \phi_j(x,y)-\sum_{\ell=1}^n h_\ell(y_\ell)$ where $f_i,h_\ell$ are convex and nonsmooth with efficiently computable proximal map and $\phi(x,y)$ is a smooth convex-concave function. The convergence rate of $\mathcal O(\log(k)/k)$ is obtained for merely convex setting by sampling the component functions using an increasing sample size. However, in this paper, we introduce a deterministic method to solve a nonsmooth optimization problem with a conic constraint. \vspace{-4mm}

\paragraph{Contribution.} In this paper, we consider a nonsmooth minimization with a conic constraint. Considering the equivalent saddle point formulation, we propose a novel primal-dual incremental gradient (PDIG) scheme. In particular, the proposed method comprises a deterministic cycle in which only two constraints, and one objective function component, $f_i$, are utilized to update the iterates. This new approach significantly improves the previous state-of-the-art incremental gradient method for constrained minimization problems \cite{kaushik2020projection} from $\mathcal O(1/k^{\frac{1}{4}})$ to $\mathcal O (1/\sqrt{k})$ in terms of suboptimality/infeasibility. Moreover, the proposed scheme guarantees a convergence rate in a deterministic manner, in contrast to randomized methods \cite{xu2020primal} where the convergence rate is in the expectation sense.

In Section \ref{sec:assump}, we provide the main assumptions and definitions, required for the convergence analysis. Next, in Section \ref{sec:rate}, we introduce  PDIG method and show the convergence rate of $\mathcal O(1/\sqrt k)$ for both suboptimality and infeasibility. Finally, in Section \ref{sec:num} we implemented the proposed algorithm to solve the constrained Lasso problem and compare it with other competitive methods. 

\section{Assumptions and definitions}\label{sec:assump}
In this section, we outline some important notations, definitions and the required assumptions that we consider for the analysis of the method. 
\paragraph{Notation.} Throughout the paper, $\|.\|$ denotes the Euclidean norm and $\mbox{\bf relint}(X)$ denotes the relative interior of the set $X$. We define $\mbox{\bf dist}_{\mathcal K}(u)\triangleq \|\mathbf \Pi_\mathcal K(u)-u\|=\||\mathbf \Pi_{-\mathcal K^*}(u)\|$. Also, $\mathbf I_d$ denotes $d\times d$ identity matrix. 
\begin{definition}
Define $U_i\in \mathbb R^{d\times d_i}$ for $i\in\{1,\hdots,m\}$ such that $\mathbf I_d=[U_1,\hdots,U_m]$. 
\end{definition}
We impose the following requirements on problem \eqref{p1}.
\begin{assumption}\label{assump1} For all $i\in \{1,\hdots,m\}$, the following hold:\\
(a) A primal-dual solution, $(x^*,y^*)$, of problem \eqref{p1} exists. \\
(b) Function $f_i$ is convex and nonsmooth. \\
(c) $f_i$ is Lipschitz continuous with constant $L$. \\
(d) $X$ is a  compact and convex set, i.e., $\exists D>0$ s.t. $\|x\|\leq D$, $\forall x\in X$. \\
(e) There exists a constant $B>0$ such that $\|y^*\|\leq B$. 
\end{assumption}

\aj{Assumption \ref{assump1}(c) is a common assumption for nonsmooth problems and it implies that $f$ at every point $x$ admits a subgradient $g(x)$ such that $\|g(x)\|\leq L$. We assume that this small norm subgradient $g(x)$ is exactly the one reported by the first-order oracle as called with input $x$ and this is not a severe restriction, since at least in the interior of the domain all subgradients of $f$ are ``small" in the outlined sense (see section 5.3 in \cite{ben2001lectures} for more details). The following lemma states an important relation required for our convergence results.  
\begin{lemma}\label{lip rel}
Suppose a convex function $f:\mathbb R^n\to \mathbb R$ is Lipschitz continuous with constant $L$. Then $f(x)\leq f(y)+g(y)^T(x-y)+2L\|x-y\|$ holds for any $x,y\in \mathbb R^n$, where $g$ is a subgradient of function $f$.
\end{lemma}
\begin{proof}
Using convexity of function $f$, Cauchy-Schwarz inequality, and boundedness of the subgradient, we can show the desired result as follows:
\begin{align*}
\langle g(y),y-x\rangle&=\langle g(y)- g(x),y-x\rangle+\langle g(x),y-x\rangle\\
&\leq \|g(y)- g(x)\|\|x-y\|+f(y)-f(x)\leq 2L\|x-y\|+f(y)-f(x).
\end{align*} 
\end{proof}}

In addition, note that the dual bound $B$ in Assumption \ref{assump1}(e) can be computed efficiently if a slater point of \eqref{p1} is available using the following lemma.  
\begin{lemma}\label{slater} \cite{hamedani2018primal}
Let $\hat x$ be a slater point of \eqref{p1}, i.e. $\hat x\in \mbox{\bf relint}(\mbox{\bf dom}(f))$ such that $Ax-b\in \mbox{int}(-\mathcal K)$, and $h:\mathbb R^d\to \mathbb R\cup \{-\infty\}$ denote the dual function, i.e.,
$$h(y)=\begin{cases}\inf_x f(x)+\langle Ax-b,y \rangle,&y\in \mathcal K^*\\-\infty,&o.w. \end{cases}$$
For any $\hat y\in \mbox{\bf dom}(h)$, let $Q_{\hat y}=\{y\in \mbox{\bf dom}(h): h(y)\geq h(\hat y)\}\subset \mathcal K^*$ denotes the corresponding superlevel set. Then for all $\hat y\in \mbox{\bf dom}(h)$, $Q_{\hat y}$ can be bounded as $\|y\|\leq \frac{f(\hat x)-h(\hat y)}{r^*}, \quad \forall y\in Q_{\hat y}$ where $0<r^*\triangleq \min_u \{-\langle A\hat x-y,u\rangle: \ \|u\|= 1, \ u\in \mathcal K^*\}$.
\end{lemma}

\section{Convergence analysis}\label{sec:rate}
In this section, we propose the Primal-Dual Incremental Gradient (PDIG) method, displayed in Algorithm \ref{alg1} to solve problem \eqref{sp}. 
\label{sec:1}
\begin{algorithm}[htbp]
\caption{Primal-Dual Incremental Gradient (PDIG) method}
\label{alg1}
{\bf input:} $x_1\in \mathbb R^n$, $y_1\in \mathbb R^d$, positive
sequences $\{\gamma_k\}_k$ and $\{\eta_k\}_k$, \aj{and let $x_{1,0}\gets x_{1}$}\\
{\bf for} $k=1\hdots K$ {\bf do}\\
\mbox{}\quad $(x_{k,1},y_{k,1})\gets (x_k,y_k)$;\\
\mbox{}\quad {\bf for} $i=	1,\hdots,m$ {\bf do}\\
\aj{\mbox{}\qquad $A_0\gets A_m$ and $U_0\gets U_m$;}\\
\mbox{}\qquad $y_{k,i+1}\gets  \mathbf \Pi_{Y\cap\mathcal K^*}\left(y_{k,i}+\eta_kU_i(A_ix_{k,i}-b_i)+\eta_kU_{i-1}A_{i-1}(x_{k,i}-x_{k,i-1})\right)$;\\
\mbox{}\qquad $x_{k,i+1}\gets\mathbf\Pi_X\left(x_{k,i}-\gamma_k(g_i(x_{k,i})+A_i^TU_i^Ty_{k,i+1})\right)$, where $g_i(x)\in \partial f_i(x)$;\\
\mbox{}\quad {\bf end for}\\
\aj{\mbox{}\quad $x_{k+1,0}\gets x_{k,m}$, $(x_{k+1},y_{k+1})\gets (x_{k,m+1},y_{k,m+1})$}; \\
{\bf end for}
\end{algorithm}

In the following theorem, we state our main result which is the convergence rate of PDIG in terms of suboptimality and infeasibility. 
\begin{theorem}\label{th1}
Suppose Assumption \ref{assump1} holds. Let $\{x_k,y_k\}_{k\geq 1}$ be the iterates generated by Algorithm \ref{alg1}, with the step-sizes chosen as $\eta_k=\tfrac{1}{ a_{\max} \sqrt k}$ and $\gamma_k=\tfrac{1}{ a_{\max}+\sqrt k}$ for all $k\geq 1$, where $a_{\max}=\max_{1\leq i\leq m}\{\|A_i\|\}$. Then the following result holds
$$\max \left\{|f(\bar x_K)-f(x^*)|, \mbox{\bf dist}_{-\mathcal K}(A\bar x_K-b)\right\}\leq \phi(\bar x_{K},\tilde y)-\phi(x^*,\bar y_{K})\leq \mathcal O(1/\sqrt K),$$
where $\tilde y\triangleq (\|y^*\|+1) \mathbf \Pi_{\mathcal K^*}(A\bar x_K-b)\| \mathbf \Pi_{\mathcal K^*}(A\bar x_K-b)\|^{-1}$ and $(\bar x_K,\bar y_K)\triangleq \tfrac{1}{K}\sum_{k=1}^K (x_k,y_k)$.
\end{theorem}

Before proving Theorem \ref{th1}, we state a technical lemma for projection mappings and then provide a one-step analysis of the algorithm in Lemma \ref{lem2}. 
\begin{lemma}\label{lem1}\cite{bertsekas2003convex}
\noindent Let $ X\subseteq \mathbb{R}^n $  be a nonempty closed and convex set. Then the following hold:
(a) $\|\Pi_X [u]- \Pi_X [v]\| \leq \|u-v\| $ for all $ u,v \in \mathbb{R}^n$;
(b) $ (\Pi_X [u]-u)^T(x-\Pi _X [u]) \geq 0 $ for all $u \in \mathbb{R}^n$ and $x \in X$.  
\end{lemma}


\begin{lemma}\label{lem2}
Suppose Assumption \ref{assump1} holds. Let $\{x_k,y_k\}_{k\geq 1}$ be the iterates generated by Algorithm \ref{alg1}, with the step-sizes chosen as $\eta_k=\tfrac{1}{ a_{\max} \sqrt k}$ and $\gamma_k=\tfrac{1}{ a_{\max}+\sqrt k}$ for all $k\geq 1$, where $a_{\max}=\max_{1\leq i\leq m}\{\|A_i\|\}$. Then the following holds for any $y\in Y\cap \mathcal K^*$.
\begin{align*}
\phi(\bar x_{ k},y)-\phi(x^*,\bar y_{ k})&\leq \tfrac{1}{ K} \left(\tfrac{1}{ a_{\max}+1}+2\sqrt K\right) (\tilde C_3+\tilde{C}_1)+\tfrac{1}{ K}\left(\tfrac{1}{ a_{\max}}+\tfrac{\sqrt K}{ a_{\max}}\right)\tilde{C}_2\\
&\quad+\tfrac{2mL^2}{ K\sqrt{K}}+\tfrac{2(B+1)^2a_{\max}\sqrt K}{ K}+\tfrac{2D^2(a_{\max}+\sqrt K)}{ K}\\
&\quad+4D^2\left(\tfrac{a_{\max}+\sqrt K}{ 2K}-\tfrac{\sqrt K}{ K}\right)\leq \mathcal O(1/\sqrt K),
\end{align*}
for some constants $\tilde C_1, \tilde C_2, \tilde C_3\geq 0$ where $(\bar x_K,\bar y_K)\triangleq \tfrac{1}{K}\sum_{k=1}^K (x_k,y_k)$.
\end{lemma}
\begin{proof}
For any $k\geq1$, we have the following for any $y\in\mathbb R^d$,
\aj{\begin{align*}
\nonumber \|y_{k,i+1}-y\|^2&=\|y_{k,i+1}-y_{k,i}\|^2+\|y_{k,i}-y\|^2+2\langle y_{k,i+1}-y_{k,i},y_{k,i}-y\pm y_{k,i+1}\rangle\\
&=\|y_{k,i+1}-y_{k,i}\|^2+\|y_{k,i}-y\|^2-2\| y_{k,i+1}-y_{k,i}\|^2\\&\quad+2\langle  y_{k,i+1}-y_{k,i},y_{k,i+1}-y\rangle\\
&=\|y_{k,i}-y\|^2-\|y_{k,i+1}-y_{k,i}\|^2 +\underbrace{2\langle y_{k,i+1}-y_{k,i},y_{k,i+1}-y\rangle}_{\text{{\tiny Term (a)}}}
\end{align*}}
From the definition of $y_{k,i+1}$ and Lemma \ref{lem1}(b) the following holds:
\begin{align}\label{bound y1}
\nonumber 0&\leq (y_{k,i+1}-(y_{k,i}+\eta_kU_i(A_ix_{k,i}-b_i)+\eta_kU_{i-1}A_{i-1}(x_{k,i}-x_{k,i-1}))^T(y-y_{k,i+1})\\
 &=(y_{k,i+1}-y_{k,i})^T(y-y_{k,i+1})+\left(\eta_kU_i(A_ix_{k,i}-b_i)+\eta_kU_{i-1}A_{i-1}(x_{k,i}-x_{k,i-1})\right)^T(y_{k,i+1}-y).
\end{align}
Therefore, term (a) can be written as \begin{align*}
&2\langle y_{k,i+1}-y_{k,i},y_{k,i+1}-y\rangle\leq 2(\eta_kU_i(A_ix_{k,i}-b_i)+\eta_kU_{i-1}A_{i-1}(x_{k,i}-x_{k,i-1}))^T(y_{k,i+1}-y).\end{align*}
Hence, we have the following:
\begin{align*}
\quad&\|y_{k,i+1}-y\|^2\\
&\leq\|y_{k,i}-y\|^2-\|y_{k,i+1}-y_{k,i}\|^2+2\big( \eta_kU_i(A_ix_{k,i}-b_i)\\
&\quad+\eta_kU_{i-1}A_{i-1}(x_{k,i}-x_{k,i-1})\pm \eta_k U_i A_ix_{k,i+1})\big)^T(y_{k,i+1}-y)\\
&=\|y_{k,i}-y\|^2-\|y_{k,i+1}-y_{k,i}\|^2+2\eta_k(y_{k,i+1}-y)^TU_i(A_ix_{k,i+1}-b_i)\\
&\quad+2\eta_k(y_{k,i+1}-y\pm y_{k,i})^TU_{i-1}(A_{i-1}(x_{k,i}-x_{k,i-1}))\\
&\quad-2\eta_k(y_{k,i+1}-y)^TU_i(A_i(x_{k,i+1}-x_{k,i}))\\
&=\|y_{k,i}-y\|^2-\|y_{k,i+1}-y_{k,i}\|^2+2\eta_k(y_{k,i+1}-y)^TU_i(A_ix_{k,i+1}-b_i)\\
&\quad +2\eta_k(y_{k,i}-y)^TU_{i-1}(A_{i-1}(x_{k,i}-x_{k,i-1}))\\
&\quad +2\eta_k(y_{k,i+1}-y_{k,i})^TU_{i-1}(A_{i-1}(x_{k,i}-x_{k,i-1}))\\
&\quad -2\eta_k(y_{k,i+1}-y)^TU_i(A_i(x_{k,i+1}-x_{k,i})),
\end{align*}
Now using Young's inequality, i.e., $a^T b\leq \tfrac{1}{ 2\alpha_k}\|a\|^2+\tfrac{\alpha_k}{ 2}\|b\|^2$, for any $a,b\in\mathbb R^d$ and $\alpha_k>0$, we conclude that
\vspace*{-1mm}
\begin{align}\label{bound1 y}
\quad&\|y_{k,i+1}-y\|^2\leq \|y_{k,i}-y\|^2+(\tfrac{\eta_k}{ \alpha_k}-1)\|y_{k,i+1}-y_{k,i}\|^2\\
&\nonumber\quad+2\eta_k(y_{k,i+1}-y)^TU_i(A_ix_{k,i+1}-b_i)\\
&\nonumber\quad +2\eta_k(y_{k,i}-y)^TU_{i-1}(A_{i-1}(x_{k,i}-x_{k,i-1}))\\
&\nonumber\quad+\eta_k\alpha_k\|U_{i-1}A_{i-1}(x_{k,i}-x_{k,i-1})\|^2-2\eta_k(y_{k,i+1}-y)^TU_i(A_i(x_{k,i+1}-x_{k,i})).
\end{align}
\aj{Similar to \eqref{bound y1}, from the update of $x_{k,i+1}$ and Lemma \ref{lem1}(b) the following holds:
\begin{align*}
(x_{k,i+1}-x_{k,i})^T(x_{k,i+1}-x^*)\leq \left(\gamma_k(g_i(x_{k,i})+A_i^TU_i^Ty_{k,i+1})\right)^T(x^*-x_{k,i+1}).
\end{align*}
Therefore, one can conclude that
\begin{align*}
\|x_{k,i+1}-x^*\|^2&=\|x_{k,i+1}-x_{k,i}\|^2+\|x_{k,i}-x^*\|^2+2\langle x_{k,i+1}-x_{k,i},x_{k,i}-x^*\pm x_{k,i+1}\rangle\\
&=\|x_{k,i}-x^*\|^2-\|x_{k,i+1}-x_{k,i}\|^2+2\langle x_{k,i+1}-x_{k,i},x_{k,i+1}-x^* \rangle\\
&\leq \|x_{k,i}-x^*\|^2-\|x_{k,i+1}-x_{k,i}\|^2 +2\left(\gamma_k(g_i(x_{k,i})+A_i^TU_i^Ty_{k,i+1})\right)^T\underbrace{(x^*-x_{k,i+1})}_{\text{{\tiny Term (b)}}}
\end{align*}
Indeed, adding and subtracting $x_{k,i}$ to term (b) leads to
\begin{align}\label{in1}
\notag\|x_{k,i+1}-x^*\|^2&\leq \|x_{k,i}-x^*\|^2-\|x_{k,i}-x_{k,i+1}\|^2-2\gamma_k(x_{k,i+1}-x_{k,i})^T(g_i(x_{k,i})+A_i^TU_i^Ty_{k,i+1})\\ &\quad-2\gamma_k(x_{k,i}-x^*)(g_i(x_{k,i})+A_i^TU_i^Ty_{k,i+1}).
\end{align}}
From Assumption \ref{assump1}(b), one can easily show that $-2\gamma_k(x_{k,i}-x^*)^Tg_i(x_{k,i})\leq -2\gamma_k(x_{k,i}-x^*)^T(f_i(x_{k,i})-f_i(x^*))$ and from Assumption \ref{assump1}(c) and Lemma \ref{lip rel}, we have that $-2\gamma_k(x_{k,i+1}-x_{k,i}^T)g_i(x_{k,i})\leq 2\gamma_k(f_i(x_{k,i})-f_i(x_{k,i+1}))+4\gamma_kL\|x_{k,i+1}-x_{k,i}\|$. Moreover, for some $\beta_k>0$ we know that $4\gamma_kL\|x_{k,i+1}-x_{k,i}\| \leq\tfrac{2 \gamma_kL^2}{ \beta_k}+2\gamma_k\beta_k\|x_{k,i+1}-x_{k,i}\|^2$, hence, \eqref{in1} can be written as follows.
\begin{align}\label{bound x}
\nonumber\|x_{k,i+1}-x^*\|^2&\leq \|x_{k,i}-x^*\|^2-\|x_{k,i}-x_{k,i+1}\|^2\\
&\nonumber +2\gamma_k(f(x_{k,i})-f_i(x_{k,i+1}))+\tfrac{2\gamma_kL^2}{ \beta_k}+2\gamma_k\beta_k\|x_{k,i+1}-x_{k,i}\|^2\\
&\nonumber -2\gamma_k(x_{k,i+1}-x^*)^TA_i^TU_i^Ty_{k,i+1}-2\gamma_k(f_i(x_{k,i})-f_i(x^*))\\
\nonumber&=\|x_{k,i}-x^*\|^2-(2\gamma_k\beta_k-1)\|x_{k,i}-x_{k,i+1}\|^2+2\gamma_k(f_i(x^*)-f_i(x_{k,i+1}))\\
& +\tfrac{2\gamma_kL^2}{ \beta_k}-2\gamma_k(x_{k,i+1}-x^*)^TA_i^TU_i^Ty_{k,i+1}.
\end{align}
Multiplying \eqref{bound1 y} by $\tfrac{1}{ 2\eta_k}$, \eqref{bound x} by $\tfrac{1}{ 2\gamma_k}$ and summing up the result leads to 
\begin{align*}
&f_i(x_{k,i+1})-f_i(x^*)-(y_{k,i+1}-y)^TU_i(A_ix_{k,i+1}-b_i)- (x^*-x_{k,i+1})^TA_i^TU_i^Ty_{k,i+1}\\
&\quad \leq\tfrac{1}{ 2\eta_k}\|y_{k,i}-y\|^2-\tfrac{1}{ 2\eta_k}\|y_{k,i+1}-y\|^2\\
&\qquad +\left(\tfrac{1}{ 2\alpha_k}-\tfrac{1}{ 2\eta_k}\right)\|y_{k,i+1}-y_{k,i}\|^2+\tfrac{1}{ 2\gamma_k}\|x_{k,i}-x^*\|^2-\tfrac{1}{ 2\gamma_k}\|x_{k,i+1}-x^*\|^2\\
&\qquad+\left(\beta_k-\tfrac{1}{ 2\gamma_k}\right)\|x_{k,,i}-x_{k,i+1}\|^2+(y_{k,i}-y)^TU_{i-1}\left[A_{i-1}(x_{k,i}-x_{k,i-1})\right]\\
&\qquad +\tfrac{\alpha_k}{ 2}\|U_{i-1}A_{i-1}(x_{k,i}-x_{k,i-1})\|^2-(y_{k,i+1}-y)^T\left[U_iA_i(x_{k,i+1}-x_{k,i})\right]+\tfrac{L^2}{ \beta_k}.
\end{align*}
Next, by selecting $\eta_k$ and $\gamma_k$ such that $\tfrac{1}{ 2\alpha_k}-\tfrac{1}{ 2\eta_k}\leq 0$ and $\tfrac{\alpha_k}{2}\|U_{i-1}A_{i-1}\|^2\leq \tfrac{1}{ 2\gamma_k}-{\beta_k}$, one can drop $\left(\tfrac{1}{ 2\alpha_k}-\tfrac{1}{ 2\eta_k}\right)\|y_{k,i+1}-y_{k,i}\|^2 $ in the above inequality  to obtain the following result 
\begin{align*}
&f_i(x_{k,i+1})-f_i(x^*)-(y_{k,i+1}-y)^TU_i(A_ix_{k,i+1}-b_i)- (x^*-x_{k,i+1})^TA_i^TU_i^Ty_{k,i+1}\\
&\quad \leq \tfrac{1}{ 2\eta_k}\left(\|y_{k,i}-y\|^2-\|y_{k,i+1}-y\|^2\right)+\tfrac{1}{ 2\gamma_k}\left(\|x_{k,i}-x^*\|^2-\|x_{k,i+1}-x^*\|^2\right)\\
&\qquad+(y_{k,i}-y)^T\left[U_{i-1}A_{i-1}(x_{k,i}-x_{k,i-1})\right]-(y_{k,i+1}-y)^T\left[U_iA_i(x_{k,i+1}-x_{k,i})\right]\\
&\qquad+\left({\beta_k}-\tfrac{1}{ 2\gamma_k}\right)\left(\|x_{k,i}-x_{k,i+1}\|^2-\|x_{k,i-1}-x_{k,i}\|^2\right)+\tfrac{L^2}{ \beta_k}.
\end{align*}
Summing the result over $i$ from $1$ to $m$ we conclude that
\begin{align}\label{bound 1}
\nonumber&\sum_{i=1}^m\left(f_i(x_{k,i+1})-f_i(x^*)-(y_{k,i+1}-y)^TU_i(A_ix_{k,i+1}-b_i)\right)- (x^*-x_{k,i+1})^TA_i^TU_i^Ty_{k,i+1}\\
\nonumber&\quad \leq \tfrac{1}{ 2\eta_k}\left(\|y_{k,1}-y\|^2-\|y_{k,m+1}-y\|^2\right)+\tfrac{1}{ 2\gamma_k}\left(\|x_{k,1}-x^*\|^2-\|x_{k,m+1}-x^*\|^2\right)\\
\nonumber&\qquad +(y_{k,1}-y)^T\left[U_{0}A_{0}(x_{k,1}-x_{k,0})\right]-(y_{k,m+1}-y)^T\left[U_mA_m(x_{k,m+1}-x_{k,m})\right]\\
&\qquad+\left({\beta_k}-\tfrac{1}{ 2\gamma_k}\right)\left(\|x_{k,m}-x_{k,m+1}\|^2-\|x_{k,0}-x_{k,1}\|^2\right)+\tfrac{mL^2}{ \beta_k}.
\end{align}
Now we proceed by providing a lower bound for the left hand side of \eqref{bound 1}. In particular, using Lemma \ref{lem1}(a) the following holds:
\begin{align*}
\|x_{k,2}-x_k\|&=\|\mathbf \Pi_X(x_{k,1}-\gamma_k(g_1(x_{k,1})+A_1^TU_1^Ty_{k,3}))-\mathbf \Pi_X(x_k)\|\\
&\leq \gamma_k(\|g_1(x_{k,1})\|+\|A_1\|\|U_1^Ty_{k,3}\|)\leq \gamma_k(L+(B+1)\|A_1\|).\\
\Rightarrow\|x_{k,3}-x_k\|&=\|\mathbf \Pi_X(x_{k,2}-\gamma_k(g_2(x_{k,2})+A_2^TU_2^Ty_{k,4}))-\mathbf \Pi_X(x_k)\|\\
&\leq \|x_{k,2}-x_k\|+\gamma_k(\|g_2(x_{k,2})\|+\|A_2\|\|U_2^Ty_{k,4}\|)\\&\leq  \gamma_k(2L+(B+1)(\|A_1\|+\|A_2\|)).
\end{align*}
Therefore, continuing this procedure for any $i\geq 1$, we conclude that 
\begin{align}\label{bound 4}
\|x_{k,i}-x_k\|\leq \gamma_k\left(iL+(B+1)\sum_{\ell=1}^i\|A_\ell\|\right)\leq \gamma_ki\left(L+(B+1)a_{\max}\right),
\end{align}
where $a_{\max}=\max_{1\leq i\leq m}\{\|A_i\|\}$. Let $b_{\max}=\max_{1\leq i\leq m}\{\|b_i\|\}$, then similar to \eqref{bound 4} one can also obtain the following for the dual iterates
\begin{align}\label{bound y}
\|y_{k,i}-y_k\|\leq \eta_k i(3Da_{\max}+b_{\max}).
\end{align}
Now, we can obtain a lower bound for the left hand side of \eqref{bound 1}. 
\begin{align}\label{bound 2}
\nonumber&\sum_{i=1}^m\left(f_i(x_{k,i+1})-f_i(x^*)-(y_{k,i+1}-y)^TU_i(A_ix_{k,i+1}-b_i)\right)- (x^*-x_{k,i+1})^TA_i^TU_i^Ty_{k,i+1}\\
\nonumber&\quad =f(x_k)-f(x^*)+\sum_{i=1}^m\left(f_i(x_{k,i+1})-f_i(x_k)-(y_{k,i+1}-y)^TU_i(A_ix_{k,i+1}-b_i)\right)\\\nonumber&\qquad- (x^*-x_{k,i+1})^TA_i^TU_i^Ty_{k,i+1}\\
\nonumber&\quad \geq f(x_k)-f(x^*)+\sum_{i=1}^m L\|x_{k,i+1}-x_k\|+\sum_{i=1}^m \big[(y-y_{k,i+1})^TU_i(A_ix_{k,i+1}-b_i)\\
&\qquad- (x^*-x_{k,i+1})^TA_i^TU_i^Ty_{k,i+1}\big] \pm [ (y-y_k)^T(Ax_k-b)-y_k^TA(x^*-x_k)],
\end{align}
where in the last inequality we used Assumption \ref{assump1}(c). Also, using \eqref{bound 4} and \eqref{bound y}, we can obtain the following bound for any $y\in Y\cap \mathcal K^*$,
\begin{align}\label{bound 3}
\nonumber\quad&\Big|\sum_{i=1}^m \big[(y-y_{k,i+1})^TU_i(A_ix_{k,i+1}-b_i)\big]- (y-y_k)^T(Ax_{k}-b)
\\&\quad \nonumber+\sum_{i=1}^m \big[y_{k,i+1}^TU_iA_i(x_{k,i+1}-x^*)\big]-y_k^TA(x_k-x^*)\Big|\\ \nonumber&=
\Big|\sum_{i=1}^m \big[y^TU_i(A_ix_{k,i+1}-b_i)+y_{k,i+1}^TU_ib_i\big]-y^T(Ax_k-b)-y_k^Tb\\&\quad \nonumber-\sum_{i=1}^m \big[y_{k,i+1}^TU_iA_ix^*\big]+y_k^TAx^* \Big|\\
\nonumber&=\Big| \sum_{i=1}^m \big[y^TU_i(A_ix_{k,i+1}-b_i)\big]-y^T(Ax_k-b)\\&\quad \nonumber-\sum_{i=1}^m \big[y_{k,i+1}^TU_i^T(A_ix^*-b_i)\big]+y_k^T(Ax^*-b)\Big|\\
\nonumber&\leq \left| {y}^T\left(\sum_{i=1}^m U_iA_i(x_{k,i+1}-x_{k})\right)-\left(\sum_{i=1}^m(y_{k,i+1}-y_{k})^TU_i(A_ix^*-b_i)\right) \right|\\
\nonumber& \leq \sum_{i=1}^m (B+1)\|A_i\|\|x_{k,i+1}-x_{k}\|+\sum_{i=1}^m\|y_{k,i+1}-y_{k}\|\|U_i(A_ix^*-b_i)\|\\
& \leq \gamma_k\frac{m(m+3)}{2}\|{y}\|a_{\max}(L+a_{\max}(B+1))\\
&\nonumber\quad +\eta_k\frac{m(m+3)}{2}(a_{\max}\|{x^*}\|+b_{\max})(3Da_{\max}+b_{\max})=\gamma_k\tilde{C}_1+\eta_k\tilde{C}_2,
\end{align}
where we used the fact that $y\in Y$, i.e. $\|y\|\leq B+1$, and we let  $\tilde C_1\triangleq \frac{m(m+3)}{2}(B+1)a_{\max}(L+a_{\max}(B+1))$ and $\tilde C_2\triangleq\frac{m(m+3)}{2}(a_{\max}\|{x^*}\|+b_{\max})(3Da_{\max}+b_{\max})$. Using \eqref{bound 3} within \eqref{bound 2} one can conclude the  following for any $y\in Y\cap \mathcal K^*$,
\begin{align*}
&\sum_{i=1}^m\left(f_i(x_{k,i+1})-f_i(x^*)-(y_{k,i+1}-y)^TU_i(A_ix_{k,i+1}-b_i)\right)- (x^*-x_{k,i+1})^TA_i^TU_i^Ty_{k,i+1}\\
&\quad \geq \phi(x_k,y)-\phi(x^*,y_k)-\tfrac{m(m+1)}{ 2}\gamma_kL(L+a_{\max}(B+1))-\gamma_k\tilde{C}_1-\eta_k\tilde{C}_2.
\end{align*}
Therefore, the inequality \eqref{bound 1} can be rewritten as follows for any $y\in Y\cap \mathcal K^*$,
\begin{align}\label{bound phi}
\nonumber\quad &\phi(x_k,y)-\phi(x^*,y_k)\\
\nonumber&\leq \gamma_k \tilde C_3+\gamma_k\tilde{C}_1+\eta_k\tilde{C}_2+\tfrac{1}{ 2\eta_k}\left(\|y_{k,1}-y\|^2-\|y_{k+1,1}-y\|^2\right)\\
\nonumber&\quad+\tfrac{1}{ 2\gamma_k}\left(\|x_{k,1}-x^*\|^2-\|x_{k+1,1}-x^*\|^2\right) +(y_{k,1}-y)^T\left[U_{m}A_{m}(x_{k,1}-x_{k,0})\right]\\
\nonumber&\quad-(y_{k+1,1}-y)^T\left[U_mA_m(x_{k+1,1}-x_{k+1,0})\right]\\
&\quad+\left({\beta_k}-\tfrac{1}{ 2\gamma_k}\right)\left(\|x_{k+1,0}-x_{k+1,1}\|^2-\|x_{k,0}-x_{k,1}\|^2\right)+\tfrac{mL^2}{ \beta_k},
\end{align}
where $\tilde C_3\triangleq \tfrac{m(m+1)}{ 2}L(L+a_{\max}(B+1))$ and we used $y_{k,m+1}=y_{k+1,1}$, $x_{k,m+1}=x_{k+1,1}$, $x_{k,m}=x_{k+1,0}$, $A_0=A_m$, and $U_0=U_m$. Before summing \eqref{bound phi} over $k$, we state some helpful inequalities on the consecutive terms involved in \eqref{bound phi}.
\begin{align}\label{tele1}
\nonumber \quad&\sum_{k=1}^K \tfrac{1}{ 2\gamma_k}(\|x_{k,1}-x^*\|^2-\|x_{k+1,1}-x^*\|^2)\\
\nonumber &\quad=\tfrac{1}{ 2\gamma_1}\|x_{1,1}-x^*\|^2+\Big[\sum_{k=2}^K(\tfrac{1}{ 2\gamma_k}-\tfrac{1}{ 2\gamma_{k-1}})4D^2\Big]-\tfrac{1}{ 2\gamma_K}\|x_{K+1,1}-x^*\|^2\\
 &\quad\leq \tfrac{1}{ 2\gamma_K}(4D^2-\|x_{K+1,1}-x^*\|^2),
\end{align}
where we used Assumption \ref{assump1}(d) and $\{\gamma_k\}_k$ is a decreasing sequence. 
Similarly,
\begin{align}\label{tele2}
 &\sum_{k=1}^K \tfrac{1}{ 2\eta_k}(\|y_{k,1}-y\|^2-\|y_{k+1,1}-y\|^2)\leq \tfrac{1}{ 2\eta_K}(4(B+1)^2-\|y_{K+1,1}-y\|^2).
\end{align}
Summing both sides of \eqref{bound phi} over $k$ from $1$ to $K$ and using \eqref{tele1} and \eqref{tele2} we conclude that  for any $y\in Y\cap \mathcal K^*$,
\begin{align*}
&\sum_{k=1}^K \phi(x_k,y)-\phi(x^*,y_k)\\
&\leq \sum_{k=1}^K \left(\gamma_k (\tilde C_3+\tilde{C}_1)+\eta_k\tilde{C}_2+\tfrac{mL^2}{ \beta_k}\right)+\tfrac{1}{ 2\eta_K}(4(B+1)^2-\|y_{K+1,1}-y\|^2)\\
&\quad+\tfrac{1}{ 2\gamma_K}(4D^2-\|x_{K+1,1}-x^*\|^2)-(y_{K,m+1}-y)^T(U_mA_m(x_{K,m+1}-x_{K,m}))\\
&\quad +\left(\tfrac{1}{ 2\gamma_K}-{\beta_K}\right)\left(4D^2-\|x_{K+1,0}-x_{K+1,1}\|^2\right)\\
& \leq \sum_{k=1}^K \left(\gamma_k (\tilde C_3+\tilde{C}_1)+\eta_k\tilde{C}_2+\tfrac{mL^2}{ \beta_k}\right)+\tfrac{1}{ 2\eta_K}(4(B+1)^2-\|y_{K+1,1}-y\|^2)\\
&\quad+\tfrac{1}{ 2\gamma_K}(4D^2-\|x_{K+1,1}-x^*\|^2)+\tfrac{1}{ 2\alpha_K}\|y_{K,m+1}-y\|^2\\
&\quad +\tfrac{\alpha_K}{ 2}\|U_mA_m(x_{K,m+1}-x_{K,m})\|^2+\left(\tfrac{1}{ 2\gamma_K}-{\beta_K}\right)\left(4D^2-\|x_{K+1,0}-x_{K+1,1}\|^2\right).
\end{align*}
Now using the fact that $\tfrac{\alpha_K}{ 2}\|U_mA_m\|^2\leq \tfrac{1}{ 2\gamma_K}-{\beta_K}$, $x_{k,m+1}=x_{k+1,1}$, $x_{k,m}=x_{k+1,0}$ and choosing $\eta_k=\alpha_k$, one can simplify the above inequality as follows
\begin{align*}
&\sum_{k=1}^K \phi(x_k,y)-\phi(x^*,y_k)\leq \sum_{k=1}^K \left(\gamma_k (\tilde C_3+\tilde{C}_1)+\eta_k\tilde{C}_2+\tfrac{mL^2}{ \beta_k}\right)+\tfrac{2(B+1)^2}{ \eta_K}+\tfrac{2D^2}{ \gamma_K}+4D^2\left(\tfrac{1}{ 2\gamma_K}-{\beta_K}\right).
\end{align*}
Choosing $\eta_k=\alpha_k=\tfrac{1}{ a_{\max}\sqrt{k}}$, $\beta_k=\frac{1}{2}\sqrt{k}$, $\gamma_k=\tfrac{1}{ a_{\max}+\sqrt k}$ and using the fact that $\sum_{k=1}^K  \tfrac{1}{ \sqrt k}\leq1+\int_{x=1}^K \tfrac{1}{ \sqrt x}dx\leq 1+\sqrt K$ and similarly $\sum_{k=1}^K \tfrac{1}{ a_{\max}+\sqrt k}\leq \tfrac{1}{ a_{\max}+1}+2\sqrt K$, the desired result can be obtained.   

\end{proof}

Now we are ready to prove the main result of the paper. 

\paragraph{ Proof of Theorem \ref{th1}.}
Let $(x^*,y^*)$ be an arbitrary saddle point of \eqref{sp}. 
Using the fact that for any $u\in \mathbb R^d$, $u=\mathbf \Pi_{-\mathcal K}(u)+\mathbf \Pi_{\mathcal K^*}(u)$ and $\langle\mathbf \Pi_{-\mathcal K}(u),\mathbf \Pi_{\mathcal K^*}(u) \rangle=0$, one can show that $\langle A\bar x_K-b,\tilde y \rangle=(\|y^*\|+1) \mbox{\bf dist}_{-\mathcal K}(A\bar x_K-b).$
Therefore, $\phi(\bar x_{K},\tilde y)=f(\bar x_K)+(\|y^*\|+1) \mbox{\bf dist}_{-\mathcal K}(A\bar x_K-b)$. Moreover, since $(x^*,y^*)$ is a saddle point of \eqref{sp} one can conclude that $f(x^*)=\phi(x^*,y^*)\geq \phi(x^*,\bar y_K)$ and by Lemma \ref{lem2} at $y=\tilde y\in Y\cap\mathcal K^*$,
\begin{align}\label{main1}
f(\bar x_K)-f(x^*)+(\|y^*\|+1) \mbox{\bf dist}_{-\mathcal K}(A\bar x_K-b)&\leq \phi(\bar x_{K},\tilde y)-\phi(x^*,\bar y_{K})\leq \mathcal O(1/\sqrt K).
\end{align}
In addition, using the fact that for any $y\in \mathbb R^d$, $\langle y^*,y\rangle\leq \langle y^*,\mathbf \Pi_{\mathcal K^*}(y) \rangle\leq \|y^*\| \mbox{\bf dist}_{-\mathcal K}(y)$, the following can be obtain:
\begin{align}\label{main2}
  0\leq \phi(\bar x_{K},y^*)-\phi(x^*, y^*)&\nonumber=f(\bar x_K)-f(x^*)+\langle A\bar x_K-b,y^*\rangle  \\ &\leq f(\bar x_K)-f(x^*)+\|y^*\| \mbox{\bf dist}_{-\mathcal K}(A\bar x_K-b).
\end{align}
Combining \eqref{main1} and \eqref{main2} gives the desired result.

\section{Numerical results}\label{sec:num}
In this section, we compare the performance of PDIG with aIR-IG \cite{kaushik2020projection} and PDSG \cite{xu2020primal} to solve the following constrained Lasso problem. 
\begin{align}\label{lasso}
\min_{x\in [-10,10]} \tfrac{1}{2}\sum_{i=1}^m\|C_ix-d_i\|^2+\tfrac{\lambda}{m}\sum_{i=1}^m \|x\|_1, \ \mbox{s.t.} \ Bx\leq 0,
\end{align} 
where matrix $C=[C_i]_{i=1}^m\in \mathbb R^{pm\times n}$, $d=[d_i]_{i=1}^m\in \mathbb R^{pm}$, and $B\in \mathbb R^{n-1\times n}$. We set $m=1000$, $n=40$, $p=n+5$, and $\lambda=0.1$. \eqref{lasso} is a special case of \eqref{p1}, if we set $f_i(x)=\tfrac{1}{ 2}\|C_ix-d_i\|^2+\tfrac{\lambda}{m} \|x\|_1$, $A_i=B_i$ for $1\leq i\leq n-1$ and $A_i=0$ for $n\leq i\leq m$, $b_i=0$, and $\mathcal K_i=\mathbb R_+$ for all $i\in\{1,\hdots,m\}$. The problem data is generated as follows. First, we generate a vector $\bar{x}\in\mathbb R^n$ whose first 10 and last 10 components are sampled from $[-10,0]$ and $[0,10]$ uniformly at random in ascending order, respectively, and the other 20 middle components are set to zero. Next, we set $d=C\bar{x}+\eta$, where $\eta\in \mathbb R^{pm}$ is a random vector with i.i.d. components with Gaussian distribution with mean zero and standard deviation $10^{-1}$. We choose the stepsizes of PDIG as suggested in Theorem \ref{th1}. For aIR-IG, according to \cite{kaushik2020projection}, the stepsize is set to $1/(1+\sqrt k)$ and the regularizer is $10/(1+k)^{0.25}$ and for PDSG, as suggested in \cite{xu2020primal}, the primal and dual step sizes are set to $1/(log(k+1)\sqrt{k+1})$. 
\begin{figure}[htb]
\centering
\includegraphics[scale=0.13]{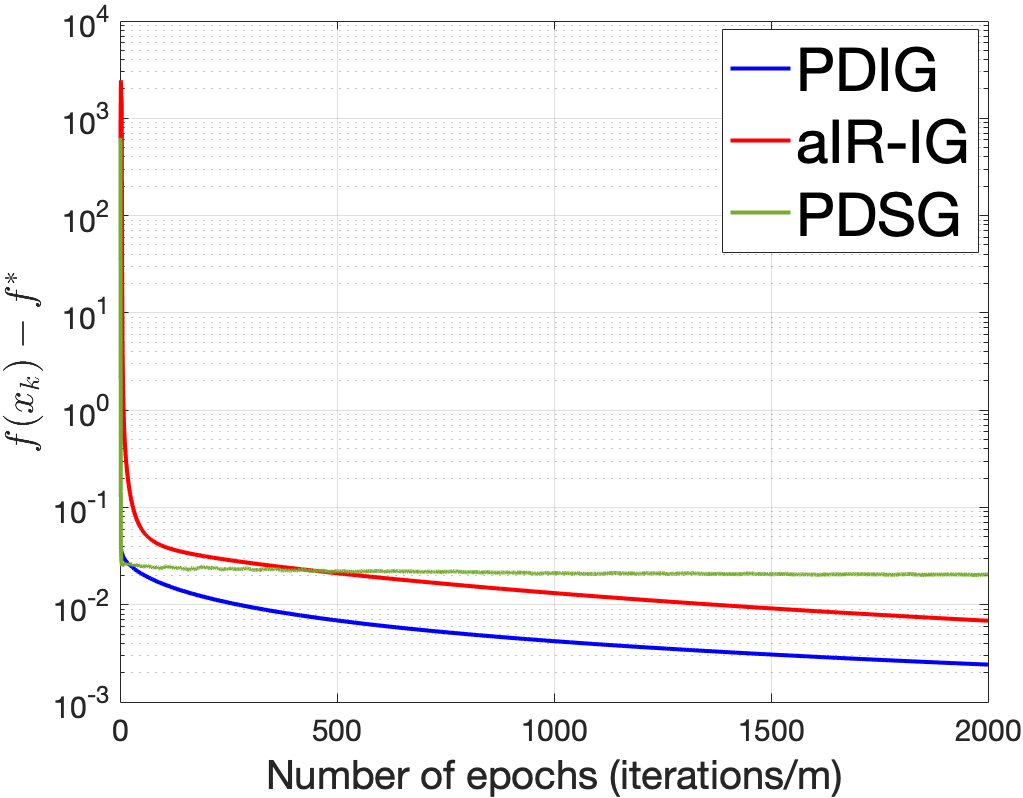}
\includegraphics[scale=0.13]{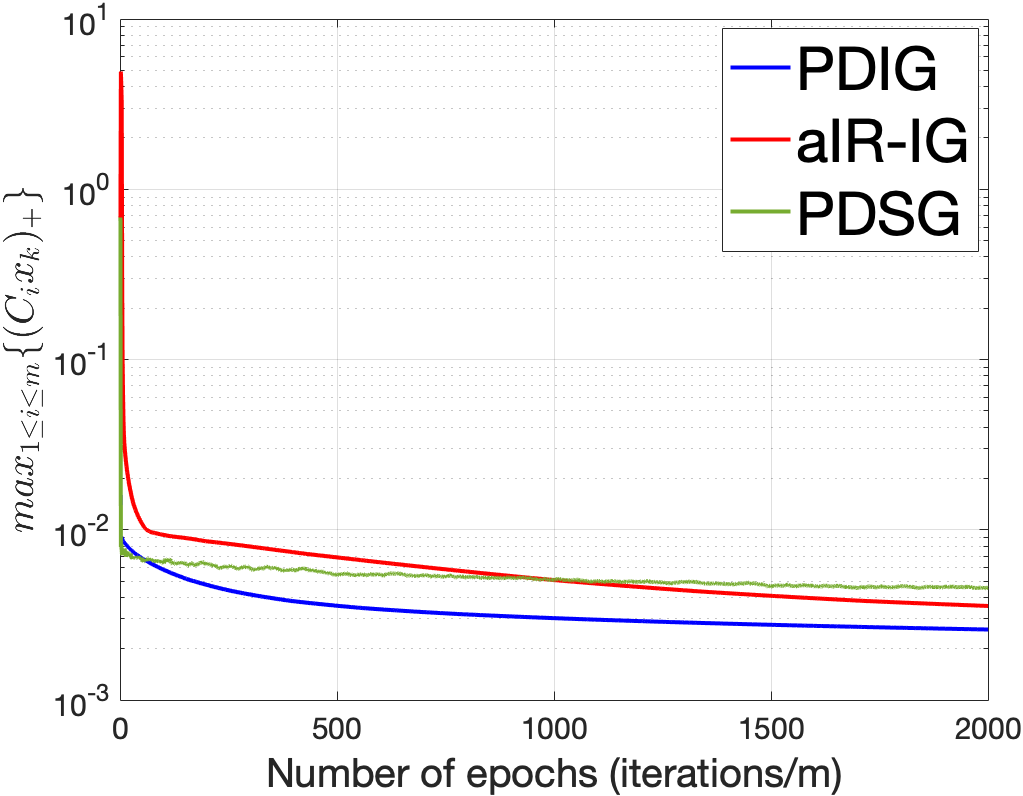}
\caption{Comparing suboptimality (left) and infeasibility (right) of PDIG, aIR-IG and PDSG.}
\label{fig1}\end{figure}

In Figure \ref{fig1}, we compared the suboptimality and infeasibility of three methods. We observe that PDIG outperforms aIR-IG which matches with the faster convergence rate of PDIG. Also, the rate of $\mathcal O( 1/\sqrt k)$ for our proposed method  is  deterministic and our step-sizes diminish periodically, in contrast with PDSG where the step-sizes diminish with iteration counter.  
\section{Concluding remarks}

Motivated by the finite sum constrained problems arising in machine learning and wireless sensor networks, we introduced a novel primal-dual incremental gradient scheme to solve nonsmooth and convex problems  with linear conic constraints. We improved the existing  rate results of the incremental gradient approach for this setting to $\mathcal O(1/\sqrt k)$ in terms of suboptimality and infeasibility in a deterministic manner.

%
%

\bibliographystyle{spmpsci}    
\bibliography{biblio} 

\begin{thebibliography}{10}
\providecommand{\url}[1]{{#1}}
\providecommand{\urlprefix}{URL }
\expandafter\ifx\csname urlstyle\endcsname\relax
  \providecommand{\doi}[1]{DOI~\discretionary{}{}{}#1}\else
  \providecommand{\doi}{DOI~\discretionary{}{}{}\begingroup
  \urlstyle{rm}\Url}\fi

\bibitem{amini2019iterative}
Amini, M., Yousefian, F.: An iterative regularized incremental projected
  subgradient method for a class of bilevel optimization problems.
\newblock In: 2019 American Control Conference (ACC), pp. 4069--4074. IEEE
  (2019)

\bibitem{bauschke1996projection}
Bauschke, H.H.: Projection algorithms and monotone operators.
\newblock Ph.D. thesis, Theses (Dept. of Mathematics and Statistics)/Simon
  Fraser University (1996)

\bibitem{ben2001lectures}
Ben-Tal, A., Nemirovski, A.: Lectures on modern convex optimization: analysis,
  algorithms, and engineering applications.
\newblock SIAM (2001)

\bibitem{bertsekas2003convex}
Bertsekas, D., Nedic, A., Ozdaglar, A.: Convex analysis and optimization.
\newblock Athena Scientific optimization and computation series. Athena
  Scientific  (2003)

\bibitem{blatt2007convergent}
Blatt, D., Hero, A.O., Gauchman, H.: A convergent incremental gradient method
  with a constant step size.
\newblock SIAM Journal on Optimization \textbf{18}(1), 29--51 (2007)

\bibitem{chambolle2018stochastic}
Chambolle, A., Ehrhardt, M.J., Richt{\'a}rik, P., Schonlieb, C.B.: Stochastic
  primal-dual hybrid gradient algorithm with arbitrary sampling and imaging
  applications.
\newblock SIAM Journal on Optimization \textbf{28}(4), 2783--2808 (2018)

\bibitem{chen1994basis}
Chen, S., Donoho, D.: Basis pursuit.
\newblock In: Proceedings of 1994 28th Asilomar Conference on Signals, Systems
  and Computers, vol.~1, pp. 41--44. IEEE (1994)

\bibitem{defazio2014saga}
Defazio, A., Bach, F., Lacoste-Julien, S.: Saga: A fast incremental gradient
  method with support for non-strongly convex composite objectives.
\newblock In: Advances in neural information processing systems, pp. 1646--1654
  (2014)

\bibitem{donoho2006compressed}
Donoho, D.L.: Compressed sensing.
\newblock IEEE Transactions on information theory \textbf{52}(4), 1289--1306
  (2006)

\bibitem{gaines2018algorithms}
Gaines, B.R., Kim, J., Zhou, H.: Algorithms for fitting the constrained lasso.
\newblock Journal of Computational and Graphical Statistics \textbf{27}(4),
  861--871 (2018)

\bibitem{gurbuzbalaban2017convergence}
Gurbuzbalaban, M., Ozdaglar, A., Parrilo, P.A.: On the convergence rate of
  incremental aggregated gradient algorithms.
\newblock SIAM Journal on Optimization \textbf{27}(2), 1035--1048 (2017)

\bibitem{hamedani2018primal}
Hamedani, E.Y., Aybat, N.S.: A primal-dual algorithm for general convex-concave
  saddle point problems.
\newblock arXiv preprint arXiv:1803.01401  (2018)

\bibitem{jalilzadeh2019doubly}
Jalilzadeh, A., Yazdandoost~Hamedani, E., Aybat, N.S., Shanbhag, U.V.: A
  doubly-randomized block-coordinate primal-dual method for large-scale saddle
  point problems.
\newblock arXiv pp. arXiv--1907 (2019)

\bibitem{kaushik2020projection}
Kaushik, H.D., Yousefian, F.: A projection-free incremental gradient method for
  large-scale constrained optimization.
\newblock arXiv preprint arXiv:2006.07956  (2020)

\bibitem{le2013stochastic}
Le~Roux, N., Schmidt, M., Bach, F.: A stochastic gradient method with an
  exponential convergence rate for finite training sets.
\newblock Pereira et al  (2013)

\bibitem{nedic2001incremental}
Nedic, A., Bertsekas, D.P.: Incremental subgradient methods for
  nondifferentiable optimization.
\newblock SIAM Journal on Optimization \textbf{12}(1), 109--138 (2001)

\bibitem{xu2017first}
Xu, Y.: First-order methods for constrained convex programming based on
  linearized augmented lagrangian function.
\newblock arXiv preprint arXiv:1711.08020  (2017)

\bibitem{xu2020primal}
Xu, Y.: Primal-dual stochastic gradient method for convex programs with many
  functional constraints.
\newblock SIAM Journal on Optimization \textbf{30}(2), 1664--1692 (2020)

\bibitem{yousefian2017smoothing}
Yousefian, F., Nedi{\'c}, A., Shanbhag, U.V.: On smoothing, regularization, and
  averaging in stochastic approximation methods for stochastic variational
  inequality problems.
\newblock Mathematical Programming \textbf{165}(1), 391--431 (2017)

\bibitem{yu2015doubly}
Yu, A.W., Lin, Q., Yang, T.: Doubly stochastic primal-dual coordinate method
  for regularized empirical risk minimization with factorized data.
\newblock CoRR, abs/1508.03390  (2015)

\bibitem{zhang2017stochastic}
Zhang, Y., Xiao, L.: Stochastic primal-dual coordinate method for regularized
  empirical risk minimization.
\newblock The Journal of Machine Learning Research \textbf{18}(1), 2939--2980
  (2017)

\bibitem{zhu2015adaptive}
Zhu, Z., Storkey, A.J.: Adaptive stochastic primal-dual coordinate descent for
  separable saddle point problems.
\newblock In: Joint European Conference on Machine Learning and Knowledge
  Discovery in Databases, pp. 645--658. Springer (2015)

\end{thebibliography}
\end{document}